\documentclass[10pt,journal,twocolumn,a4paper]{IEEEtran}

\makeatletter
\def\endthebibliography{%
  \def\@noitemerr{\@latex@warning{Empty `thebibliography' environment}}%
  \endlist
}
\makeatother

\usepackage[OT1]{fontenc} 
\usepackage{comment}
\usepackage{braket}
\usepackage{pict2e}
\usepackage{svg}
\usepackage{amsmath}
\usepackage{amssymb}
\usepackage{amsthm}
\usepackage{dsfont}
\usepackage{xcolor}
\usepackage{graphicx}
\usepackage{mathtools}
\usepackage{accents}
\usepackage{mdframed}
\usepackage{pgfplots}

\graphicspath{{grafics/}}

\newcommand{\CP}[2]{\mathbb{P} \mathopen{} \left[ #1 \, \middle| \, #2 \right] \mathclose{} }

\newcommand{\CE}[2]{\mathbb{E} \mathopen{} \left[ #1 \, \middle| \, #2 \right] \mathclose{} }

\newcommand{\e}[1]{\bar{#1}}

\newcommand{\qp}{q^\perp}
\newcommand{\qc}{q^\shortparallel}
\newcommand{\np}{\tilde{p}}

\newcommand{\figref}[1]{\figurename~\ref{#1}}

\allowdisplaybreaks

\DeclareFontFamily{U}{mathb}{}
\DeclareFontShape{U}{mathb}{m}{n}{
  <-5.5> mathb5
  <5.5-6.5> mathb6
  <6.5-7.5> mathb7
  <7.5-8.5> mathb8
  <8.5-9.5> mathb9
  <9.5-11.5> mathb10
  <11.5-> mathbb12
}{}

\newtheoremstyle{break}
  {\topsep}{\topsep}%
  {\itshape}{}%
  {\bfseries}{}%
  {\newline}{}%
\theoremstyle{break}
\newtheorem{theorem}{Theorem}

\newtheorem{lemma}[theorem]{Lemma}

\pgfplotsset{compat=1.14}

\pdfoutput=1

\begin{document}

%\title{throughput optimality of a Predictive Network Control in Discrete-Time Queueing Networks}
\title{Stability Results on Synchronized Queues in Discrete-Time for Arbitrary Dimension}

\author{Richard~Schoeffauer\textsuperscript{1}
        and~Gerhard~Wunder\textsuperscript{2}% <-this % stops a space
\thanks{\textsuperscript{1} richard.schoeffauer@fu-berlin.de, \textsuperscript{2} gerhard.wunder@fu-berlin.de}% <-this % stops a space
\thanks{Both authors are members of the Heisenberg Communication and Information Theory Group at Freie Universität Berlin}% <-this % stops a space
}

% The pa14pe pr headers
\markboth{Journal of \LaTeX\ Class Files,~Vol.~14, No.~8, August~2015}%
{Shell \MakeLowercase{\textit{et al.}}: Bare Demo of IEEEtran.cls for IEEE Journals}

\maketitle

\begin{abstract}
In a batch of synchronized queues, customers can only be serviced all at once or not at all, implying that service remains idle if at least one queue is empty.
We propose that a batch of $n$ synchronized queues in a discrete-time setting is quasi-stable for $n \in \{2,3\}$ and unstable for $n \geq 4$.  A correspondence between such systems and a random-walk-like discrete-time Markov chain (DTMC), which operates on a quotient space of the original state-space, is derived.
Using this relation, we prove the proposition by showing that the DTMC is transient for $n \geq 4$ and null-recurrent (hence quasi-stability) for $n \in \{2,3\}$ via evaluating infinite power sums over skewed binomial coefficients.

Ignoring the special structure of the quotient space, the proposition can be interpreted as a result of Pólya's theorem on random walks, since the dimension of said space is $d-1$.
\end{abstract}

\begin{IEEEkeywords}
Synchronized Queues, Binomial Coefficients, Power Sums
\end{IEEEkeywords}

\IEEEpeerreviewmaketitle

%\section{Wichtige Bemerkungen für michselbst}

\section{Introduction}

Conventional queueing networks, consisting of queue-server \textit{pairs}, are a well investigated system class. In those networks, each queue possesses its individual server, which handles the customers from exclusively that queue. In contrast, there exists the notion of \textit{paired} or \textit{synchronized} queues, in which a single server is responsible for multiple queues in such a way that service can only take place if a customer from each queue is present, as illustrated in \figref{fig::batch}. Synchronized queues find application in assembly lines \cite{Harrison1973, DeCuypere2014}, parallel computing \cite{Olvera-Cravioto2014} and matchmaking scenarios \cite{Gurvich2014}.

It is known that a single batch of synchronized queues with i.i.d. arrivals in a \textit{continuous-time} setting is unstable in the sense that the backlog does not converge to a stationary distribution \cite{Harrison1973}. For the special case of 2 synchronized queues, this was further investigated by \cite{Latouche1981}, who found the difference in the queues, the so-called excess, to be the deciding quantity implying instability. \cite{Latouche1981} also made a connection to null-recurrent and transient behavior of a corresponding Markov process.

\begin{figure}
	\centering
	\def\svgwidth{\linewidth}
%% Accompanies image file 'batch_of_synced_queues.pdf' (pdf, eps, ps)
\begingroup%
  \makeatletter%
  \providecommand\color[2][]{%
    \errmessage{(Inkscape) Color is used for the text in Inkscape, but the package 'color.sty' is not loaded}%
    \renewcommand\color[2][]{}%
  }%
  \providecommand\transparent[1]{%
    \errmessage{(Inkscape) Transparency is used (non-zero) for the text in Inkscape, but the package 'transparent.sty' is not loaded}%
    \renewcommand\transparent[1]{}%
  }%
  \providecommand\rotatebox[2]{#2}%
  \newcommand*\fsize{\dimexpr\f@size pt\relax}%
  \newcommand*\lineheight[1]{\fontsize{\fsize}{#1\fsize}\selectfont}%
  \ifx\svgwidth\undefined%
    \setlength{\unitlength}{249.4488189bp}%
    \ifx\svgscale\undefined%
      \relax%
    \else%
      \setlength{\unitlength}{\unitlength * \real{\svgscale}}%
    \fi%
  \else%
    \setlength{\unitlength}{\svgwidth}%
  \fi%
  \global\let\svgwidth\undefined%
  \global\let\svgscale\undefined%
  \makeatother%
  \begin{picture}(1,0.55681814)%
    \lineheight{1}%
    \setlength\tabcolsep{0pt}%
    \put(0,0){\includegraphics[width=\unitlength,page=1]{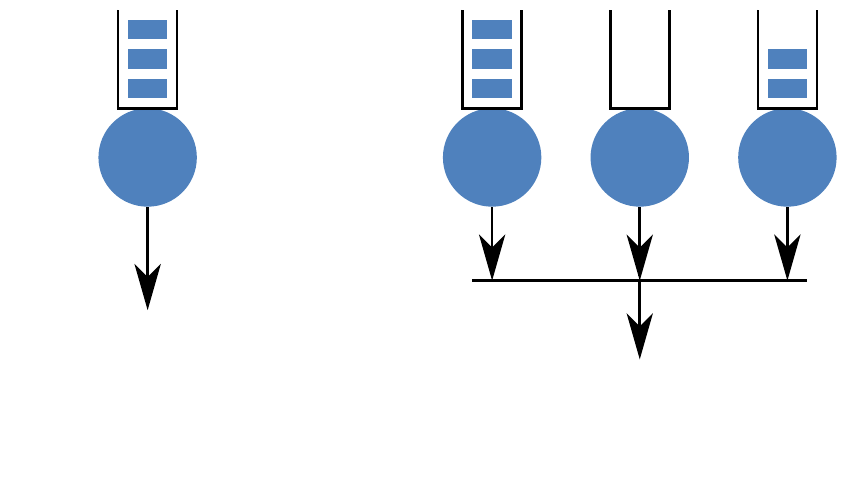}}%
    \put(0.17045453,0.14772723){\color[rgb]{0,0,0}\makebox(0,0)[t]{\lineheight{1.25}\smash{\begin{tabular}[t]{c}conventional qeue\end{tabular}}}}%
    \put(0.21712511,-0.61184304){\color[rgb]{0,0,0}\makebox(0,0)[lt]{\begin{minipage}{1.50322099\unitlength}\raggedright \end{minipage}}}%
    \put(0.54657218,0.12224497){\color[rgb]{0,0,0}\makebox(0,0)[lt]{\begin{minipage}{0.38412894\unitlength}\centering batch of 3 synchronized queues\end{minipage}}}%
    \put(0,0){\includegraphics[width=\unitlength,page=2]{batch_of_synced_queues.pdf}}%
  \end{picture}%
\endgroup%

	\caption{Graphical illustration of a batch of 3 sync. queues on the RHS; customers can not be served because the center queue is empty.}
	\label{fig::batch}
\end{figure}

However, all those results consider a batch of synchronized queues on its own and never in a network context. In networks, multiple batches might be joint together with conventional queues and a network controller has to allocate scarce resources (e.g. manpower) to each server in order to facilitate their operation. It is our goal to reach such a union and this paper presents a small, but in itself closed step towards it. In some sense, we extend the results from \cite{Latouche1981} to the multidimensional case and to a \textit{discrete-time} setting, which is our preferred setting to model queueing networks. Compared to the findings of \cite{Harrison1973}, we yield a more precise characterization of the process that is responsible for the divergence of the backlog. We show that this process is the vectorial excess and evolves on a specific quotient space whose equivalence classes are taken from the general state-space. This process resembles a random walk, can be analyzed as such, and will imply null-recurrent/quasi-stable or transient/unstable behavior, based on the number of queues in the batch.

At this point, there exists a rigorous correspondence with the infinite power sums over binomial coefficients: as we will show, the convergence/divergence of such series is equal to transient/non-transient (recurrent and null-recurrent) behavior of the excess process. Hence, this paper will investigate convergence properties of infinite power sums over skewed binomial coefficients.
Unfortunately, there are only scarce results on powers of binomial coefficients. A common approximation can be derived using the central limit theorem and yields
\begin{equation}
	\label{eq::approx}
	\binom{n}{k}\frac{1}{2^n} \sim \frac{ e^{-\frac{2}{n} \left( x-\frac{n}{2} \right)^2 } }
	{\sqrt{ \pi \frac{n}{2} }} 
	\  \ \rightarrow \ \ 
	\sum_{k=0}^n \binom{n}{k}^d \sim \frac{2^{dn}}{\sqrt{d}} \left( \frac{2}{\pi n} \right)^{\frac{d-1}{2}}
\end{equation}
However, the estimate on the LHS only holds for $k \sim \frac{n}{2}$ (where $\frac{n}{2}$ can be replaced with the mean of the normal distribution to yield a more general expression) and, for the sum on the RHS, any values with $k \gg \frac{n}{2}$ or $k \ll \frac{n}{2}$ are assumed to be negligible. Even more, for values close to $\frac{n}{2}$ the approximation error of a binomial coefficient via the central limit theorem as well as via Stirling's formula still is in the order $\mathcal{O}(\frac{1}{\sqrt{n}})$ (which follows e.g. from the Berry–Esseen theorem). This renders \eqref{eq::approx} inadequate for our investigation in the convergence of the infinite sum over $n$, since the sum over the estimation error itself could possibly be the cause for divergence. In contrast, \cite{Cusick1989} derives a method to obtain an exact recurrence relation between the binomial coefficients of arbitrary power $d$. However, this is only of limited use to us, because on one hand, our binomial coefficients will be "skewed" with powers of certain probabilities resulting from the term $\binom{n}{k}p^k q^{n-k}$, and on the other hand, these recurrences span over several past members of their sequences which makes it difficult to extract any information about convergence/divergence properties of the corresponding series.

This investigation is structured as follows: after introducing the motivating problem, we develop the correspondence between a batch of $d$ synchronized queues and a random-walk-like stochastic process. In the last part, we state and proof 2 theorems according to which the batch/stochastic process is quasi-stable/null-recurrent for $d \in \{2,3\}$ and unstable/transient for $d \geq 4$.

\section{System Model}

In the context of discrete-time queueing models, we can express a single batch of synchronized queues via the evolution of its queue state. If the batch consist of $d$ queues ($d$imensions), then the state vector will be $q_t \in \mathbb{N}^d$, where $t$ designates the time slot. Its evolution follows
\begin{equation}
	\label{eq::evolution}
	q_{t+1} = q_t - \mathbf{1} m_t v_t + a_t
\end{equation}
The tailing term represents the arrival process with $a_t \in \mathbb{N}^d$, which we will assume to be the vector of $d$ i.i.d. Bernoulli processes with parameter $p$. I.e. in each time slot and each queue, the probability of exactly one customer arriving is $p$. While the arrival represents the influx, the middle term stands for the efflux. It is $\mathbf{1}$ the vector of ones (with dimension $d$) and $v_t \in \{0,1\}$ the control vector. If $q_t \geq \mathbf{1}$, $v_t$ can be activated (set to 1), which will subtract $\mathbf{1}$ from $q_t$ in accordance with the evolution. As an additional complication, inherited from the usual network context, $(m_t)$ is a Bernoulli process with parameter $\e{m}$ that disturbs the control. Hence, even if $v_t$ is active, $m_t$ could be zero and render the control action effectless.

As mentioned, the control cannot be activated if at least one queue is empty, giving rise to the constraint
\begin{equation}
	\label{eq::constraint}
	\mathbf{1} v_t \leq q_t
\end{equation}
Obviously, an optimal control strategy activates $v_t$ whenever possible. The maximal efflux resulting from such a policy would be $\mathbf{1} \e{m}$. Hence, we must assume that $p < \e{m}$, since otherwise the influx $a_t$ would on average be greater than the efflux $-\mathbf{1} m_t v_t$, rendering the system unstable per construction. For what follows, we assume such an optimal strategy to be active.
The main contribution of this paper is the following theorem:

\begin{theorem}
\label{th::one}
A batch of $d$ coupled queues, set up as described earlier, is quasi-stable (null-recurrent) for $d \in \{2,3\}$ and unstable (transient) for $d \geq 4$.
\end{theorem}
\begin{proof}
	See sections \ref{sec::correspondence} and \ref{sec::proof}.
\end{proof}

\section{Corresponding Random Walk}

\label{sec::correspondence}

Though we defined $q_t \in \mathbb{N}^d$, we will define the state-space to be $\mathbb{Z}^d$ to obtain a vector space. Since evolution \eqref{eq::evolution} and constraint \eqref{eq::constraint} force $q_t$ to be positive, this does not impact the system at all. If the employed policy only takes into account the current system states (which is prudent to assume), then $(q_t)$ is a discrete-time Markov chain (DTMC). It is readily verified that $(q_t)$ can be segregated into two distinct processes, each of which is again a DTMC: a process perpendicular to the main diagonal of the state space, $(\qp_t)$, and a process parallel to the main diagonal, $(\qc_t)$.
\begin{equation}
\begin{gathered}
	q_{t+1} = \mathbf{1} \qc_{t+1} + \qp_{t+1}
	\\[1ex]
	\begin{aligned}
	\qc_{t+1} &:= q_{t+1} \operatorname{div} \mathbf{1} &&= \left( \qc_t + a_t \right) \operatorname{div} \mathbf{1} - m_t v_t
	\\[1ex]
	\qp_{t+1} &:= q_{t+1} \operatorname{mod} \mathbf{1} &&= \left( \qp_t + a_t \right) \operatorname{mod} \mathbf{1}
	\end{aligned}
\end{gathered}
\end{equation}
We will focus on the process $(\qp_t)$, which represents the multidimensional excess of customers, relative to the lowest queue. Crucially, $(\qp_t)$ cannot be influenced by the control.

Define $\langle \mathbf{1} \rangle := \Set{ x = \mathbf{1} \cdot k , \ k \in \mathbb{Z} }$ which is a vector subspace of $\mathbb{Z}^d$. Then we can express the state space of $(\qp_t)$ as the quotient space $\mathbb{Z}^d \slash \langle \mathbf{1} \rangle$ due to the modulo operator. Furthermore, $(\qp_t)$ behaves nearly like a random walk on ${\mathbb{Z}^d \slash \langle \mathbf{1} \rangle}$: because of $(a_t)$, the probability of an increment in each dimension is $p$. However, there are no direct decrements. Instead, based on the properties of the quotient space, a decrement in any dimension corresponds to an increment in all, except that dimension:
\begin{equation}
	\begin{pmatrix}
		-1 & 0 & \dots & 0	
	\end{pmatrix}^\intercal
	\equiv
	\begin{pmatrix}
		0 & 1 & \dots & 1	
	\end{pmatrix}^\intercal
	\mod \mathbf{1}
\end{equation}

The question of interest is, whether the DTMC $(\qp_t)$ is \textit{transient}, \textit{null-recurrent} or even \textit{recurrent}. For that purpose, let $R_d(n)$ be the probability that the process $(\qp_t)$ has $R$eturned to wherever it started after $n$ time steps: $R_d(n) := \CP{\qp_{t+n} = q^*}{\qp_t = q^*}$. Note that the original process $(q_t)$ has $d$ dimensions, while $(\qp_t)$ has $d-1$. From the theory of random walks, it is well known that the DTMC is transient if $\sum_{n=0}^\infty R_d(n)$ converges, and non-transient if the sum diverges. (This infinite sum expresses, how often the walk will return to its origin, on average.)

In our set-up, the DTMC will return to its origin after $n$ steps if every dimension experienced the same amount of increments in these $n$ time slots. Let $k$ be that amount. Then we can express the probability of every dimension experiencing $k$ increments in $n$ time slots as
\begin{equation}
	P_{n,k}^d = \left[ \binom{n}{k} p^k \left( 1-p \right)^{n-k} \right]^d
\end{equation}
(where the upper $d$ expresses the $d$-th power). Summing this over all possible $k$ yields
\begin{equation}
	\label{eq::P_def}
	R_d(n) 
	=
	\sum_{k=0}^n P_{n,k}^d
	=
	\sum_{k=0}^n \left[ \binom{n}{k} p^k \left( 1-p \right)^{n-k} \right]^d
\end{equation}

In \figref{fig::partial}, we illustrated the \textit{inverse} of $R_d(n)$ for the first few terms of the series $\sum_{n=0}^\infty R_d(n)$, when $p=\frac{1}{2}$. Intuitively, the terms for $d=2$ resemble a root-like function, implying that the sum should diverge ($\sum \frac{1}{\sqrt{n}} \to \infty$). In the same fashion, $d\geq 4$ seems to resemble a simple polynomial, which implies convergence ($\sum \frac{1}{n^2} \to \frac{\pi^2}{6}$). The case $d=3$ is problematic: it is not exactly a linear behavior, though it closely resembles one, leaving us uncertain of its convergence properties. Keep in mind, that it stands to question if these preliminary observations remain robust regarding different values of $p$.

\begin{figure}
	\centering
	\def\svgwidth{\linewidth}
	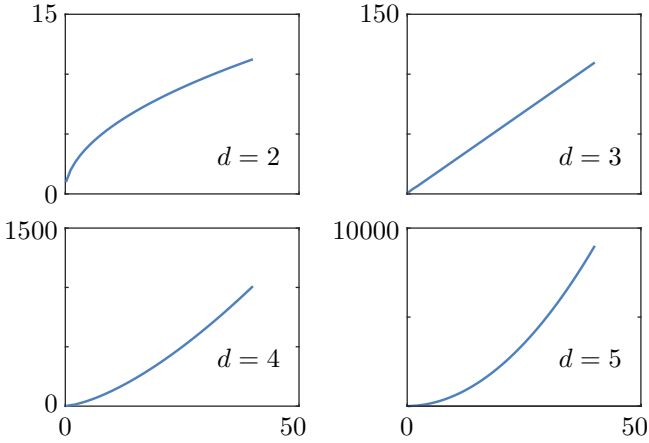
	\caption{Inverse values of $R_d(n)$ for $n=0 \dots 40$ to illustrate how fast the terms tends to 0.}
	\label{fig::partial}
\end{figure}

As an important result from that correspondence, we can reformulate Theorem~\ref{th::one}:

\begin{theorem}
For $p \in (0,1)$ and $\np = 1 - p$ the series
\begin{equation}
	\label{eq::inf_sum}
	\sum_{n=0}^\infty 
	R_d(n)
	=
	\sum_{n=0}^\infty 
	\sum_{k=0}^n
	P_{n,k}^d
	=
	\sum_{n=0}^\infty 
	\sum_{k=0}^n
	\left[
		\binom{n}{k} p^k \np^{n-k}
	\right]^d
	%\quad
	%\begin{cases}
	%	\text{diverges for } d \in \{2,3\}
	%	\\
	%	\text{converges for } d \geq 4
	%\end{cases}
\end{equation}
diverges for $d \in \{2,3\}$ and converges for $d \geq 4$.
%(For $p=q=\frac{1}{2}$, this corresponds to the sum over the normalized pascal triangle, with each entry taken to the power of $d$.)
\end{theorem}
\begin{proof}
	See section \ref{sec::proof}.
\end{proof}

\section{Proof of Main Theorem}
\label{sec::proof}

We identify $P_{n,k}^d$ with a position in Pascal's triangle: $n$ denotes the row and $k$ the column. (Especially for $p=\frac{1}{2}$ the sum \eqref{eq::inf_sum} runs over the elements of the normalized triangle, taken to the power of $d$.) For ease of notation we will use the identity $\np = 1-p $ and will treat $P_{n,k}^d$ as a function of $p$ if it suits us.
Using Sterling's approximation for the factorial, we find an upper bound for the largest terms in each row $n$, allowing us to prove convergence for $d \geq 4$. For $d = 2$ we find a lower bound that diverges. For the difficult case $d=3$ we employ a theorem by Kendall.

\begin{lemma}
\label{lemma::sum}
Let there be a set of elements $\mathcal{X} = \Set{ x_1,x_2,\dots x_n }$ with $x_i \in \mathbb{R}_+$ and  $\sum_{\mathcal{X}} x_i = 1$. Let us denote the maximal value with $\hat{x} = \max_{\mathcal{X}} x_i$ and the average with $\e{x} = \frac{1}{n} \sum_\mathcal{X} x_i$. Then it holds for each $d \in \mathbb{N}, d \geq 4$ that
\begin{equation}
	\sum_{i=1}^n x_i^d \leq \hat{x}^{d-1}
	\qquad \text{and} \qquad
	\sum_{i=1}^n x_i^2 \geq n \e{x}^2
\end{equation}
\end{lemma}
\begin{proof}
It is
\begin{equation}
	\sum_{i=1}^n x_i^d \leq \sum_{i=1}^n x_i \hat{x}^{d-1} = \hat{x}^{d-1}
\end{equation}
and
\begin{equation}
\begin{aligned}
	\sum_{i=1}^n x_i^2
	&=
	\sum_{i=1}^n ( \e{x} + [ x_i - \e{x} ] )^2
	\\
	&=
	n \e{x}^2 + 
	\underbrace{ 2\e{x} \sum_{i=1}^n [ x_i - \e{x} ]
	}_{\displaystyle =0}	
	+ \sum_{i=1}^n [ x_i - \e{x} ]^2
	\geq
	n \e{x}^2 
\end{aligned}
\end{equation}
\end{proof}

\subsection{Case $d \geq 4$}

For $d=1$ we will write $P_{n,k}$ instead of $P_{n,k}^1$ and have
\begin{equation}
	\sum_{k=0}^n P_{n,k} = 1 \qquad \forall n \in \mathbb{N}, \  \forall p\in(0,1)
\end{equation}
The terms in row $n$, which are $P_{n,1},P_{n,2} ,\dots P_{n,n}$, follow a binomial distribution, which exhibits at best 2 maxima. Denote with $\hat{P}(n)$ the largest term in row $n$, and its position with $\hat{k}(n)$. It follows that $\hat{P}(n)$ is positioned wherever the quotient of two consecutive terms is greater or equal $\{1\}$ for the first time (starting on the left):
\begin{equation}
	\frac{	\displaystyle
		\binom{n}{\hat{k}} p^{\hat{k}} \np^{n-\hat{k}} 
	}{  \displaystyle
		\binom{n}{\hat{k}+1} p^{\hat{k}+1} \np^{n-\hat{k}-1} 
	}
	\geq 1
	\quad \Longleftrightarrow \quad
	\hat{k}(n) \geq np-\np
\end{equation}
Due to the discreteness we can locate $\hat{k}(n)$ only in the interval $\hat{k}(n) = np+\delta$ with $-\np \leq \delta \leq p$. Note that choosing another value for $d$ will not change this position, since all $P_{n,k}$ are positive and will therefore only be scaled appropriately to their relative size.

Using Stirling's approximation $\sqrt{2\pi} n^{n+\frac{1}{2}} e^{-n} \leq n! \leq e n^{n+\frac{1}{2}} e^{-n}$ and some basic algebra, we get
\begin{multline}
	\hat{P}(n)
	=
	\binom{n}{np+\delta} p^{np+\delta} \np^{n \np -\delta}
	\\
	\leq
	\frac{e}{2\pi} \frac{1}{\sqrt{n p \np }}
	%\underbrace{
	\left( 1+\frac{\delta}{n p} \right)^{-n p-\delta-\frac{1}{2}}
	%}_{\displaystyle \theta_1(n)}
	%\underbrace{
	\left( 1-\frac{\delta}{n \np} \right)^{-n \np +\delta-\frac{1}{2}}
	%}_{\displaystyle \theta_2(n)}
\end{multline}
Regarding the tailing products, we can make the following estimation and find, for any given $\varepsilon \in \mathbb{R}_+$, an $N \in \mathbb{N}$ such that $\forall n \geq N$:
\begin{equation}
\begin{aligned}
    &\left( 1+ \frac{\delta}{np} \right)^{-np} \to e^{-\delta} &&\leq e^{\np} + \varepsilon
    \\
    &\left( 1+ \frac{\delta}{np} \right)^{-\delta} &&\leq 1
    \\
    &\left( 1+ \frac{\delta}{np} \right)^{-\frac{1}{2}} \to 1 &&\leq 1 + \varepsilon
\end{aligned}
\end{equation}
With slight abuse of notation concerning the values of $\varepsilon$ and $N$, this yields
\begin{equation}
	\hat{P}(n)
	\leq
	\frac{e^2 + \varepsilon}{2 \pi} \frac{1}{\sqrt{np \np}}
	\qquad
	\forall n \geq N
\end{equation}

Now, for any $d$ let $S_d$ denote the finite sum of the first $N-1$ rows, i.e.
\begin{equation}
	S_d = \sum_{n=0}^{N-1} \sum_{k=0}^n P_{n,k}^d
\end{equation}
Using lemma \ref{lemma::sum}, the entire series becomes
\begin{align}
    \nonumber
	\sum_{n=0}^{\infty} \sum_{k=0}^n
	P_{n,k}^d
	&=
	S_d + \sum_{n=N}^{\infty} \sum_{k=0}^n P_{n,k}^d
	\leq
	S_d + \sum_{n=N}^{\infty}  \hat{P}^{d-1}(n)
	\\
	&
	\leq
	S_d +  \left[ \frac{e^2+\varepsilon}{2 \pi \sqrt{p\np}} \right]^{d-1} \sum_{n=N}^\infty \frac{1}{n^{\frac{d-1}{2}}}
\end{align}
which converges for $d\geq 4$.

\subsection{Case $d=2$}

Since for $d=1$ every row adds up to one and every row consists of $n+1$ terms (resulting in an average value of $\frac{1}{n+1}$) using lemma \ref{lemma::sum} yields
\begin{equation}
	\sum_{n=0}^\infty \sum_{k=0}^n P_{n,k}^2
	\geq
	\sum_{n=0}^\infty \sum_{k=0}^n \frac{1}{(n+1)^2}
	>
	1+
	\sum_{n=1}^\infty \frac{n}{(n+n)^2}
\end{equation}
which diverges.

\subsection{Case $d=3$}

\begin{figure}
	\centering
	\def\svgwidth{88mm}
%% Accompanies image file 'geometric_kendall.pdf' (pdf, eps, ps)
\begingroup%
  \makeatletter%
  \providecommand\color[2][]{%
    \errmessage{(Inkscape) Color is used for the text in Inkscape, but the package 'color.sty' is not loaded}%
    \renewcommand\color[2][]{}%
  }%
  \providecommand\transparent[1]{%
    \errmessage{(Inkscape) Transparency is used (non-zero) for the text in Inkscape, but the package 'transparent.sty' is not loaded}%
    \renewcommand\transparent[1]{}%
  }%
  \providecommand\rotatebox[2]{#2}%
  \newcommand*\fsize{\dimexpr\f@size pt\relax}%
  \newcommand*\lineheight[1]{\fontsize{\fsize}{#1\fsize}\selectfont}%
  \ifx\svgwidth\undefined%
    \setlength{\unitlength}{249.4488189bp}%
    \ifx\svgscale\undefined%
      \relax%
    \else%
      \setlength{\unitlength}{\unitlength * \real{\svgscale}}%
    \fi%
  \else%
    \setlength{\unitlength}{\svgwidth}%
  \fi%
  \global\let\svgwidth\undefined%
  \global\let\svgscale\undefined%
  \makeatother%
  \begin{picture}(1,0.60227273)%
    \lineheight{1}%
    \setlength\tabcolsep{0pt}%
    \put(0,0){\includegraphics[width=\unitlength,page=1]{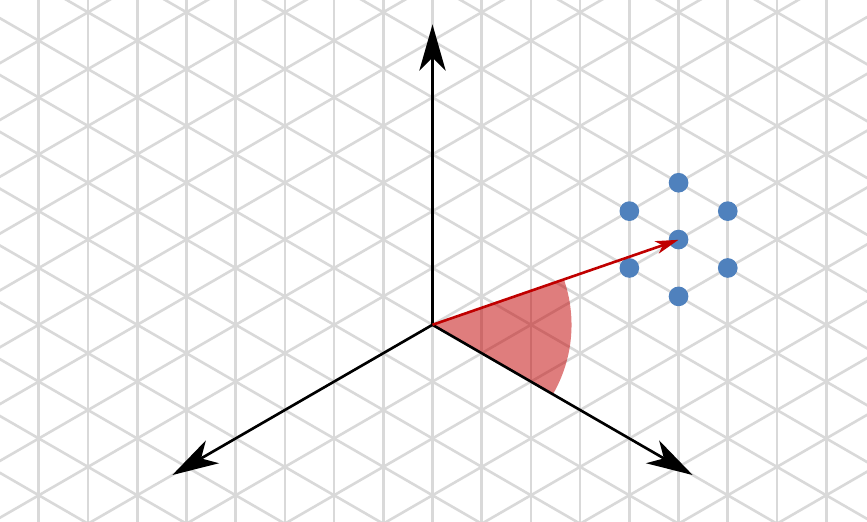}}%
    \put(0.6283503,0.31028331){\color[rgb]{0,0,0}\makebox(0,0)[lt]{\lineheight{1.25}\smash{\begin{tabular}[t]{l}$q$\end{tabular}}}}%
    \put(0.57195239,0.20811514){\color[rgb]{0,0,0}\makebox(0,0)[lt]{\lineheight{1.25}\smash{\begin{tabular}[t]{l}$\phi$\end{tabular}}}}%
  \end{picture}%
\endgroup%

	\caption{State space $\mathbb{Z}^3 \slash \langle \mathbf{1} \rangle$; red quantities describe a current queue state $\qp_t$, while blue points illustrate all possible states, that $\qp_{t+1}$ can take.}
	\label{fig::state_space}
\end{figure}

This case requires a much more intricate approach. In a first step we yield a lower bound on the series \eqref{eq::inf_sum} by setting $p=\np=\frac{1}{2}$. It is readily verified, that
\begin{equation}
\frac{d}{dp} R_d(n) = 0 \quad \text{and} \quad \frac{d^2}{dp^2} R_d(n) > 0 \qquad \text{if} \quad p=\frac{1}{2}
\end{equation}
and no other candidates for extreme points exist. Hence, it will suffice to show divergence in this symmetric case and for what follows it will be $p = \frac{1}{2}$.

We already mentioned that the series \eqref{eq::inf_sum} diverges, iff the DTMC $(\qp_t)$ is non-transient. We will now show the latter part by applying a theorem by Kendall \cite{Kendall1951}:
The DTMC $(\qp_t)$ with state-space $\mathcal{Q}$ is non-transient, iff there exists a function ${f: \mathcal{Q} \to \mathbb{R}_+}$ and a finite set $\mathcal{F} \subset \mathcal{Q}$ such that
\begin{equation}
    \label{eq::kendall_cond}
	\begin{aligned}
	\mathcal{Q}_K &:= \Set{
		q \in \mathcal{Q}:
		f(q) \leq K
	} = \text{finite for all $K$}
	\\
	\Delta f(q) &:= \CE{q_{t+1} - q_t}{q_t=q} \leq \infty \qquad \forall q \in \mathcal{Q}
	\\
	\Delta f(q) &:= \CE{q_{t+1} - q_t}{q_t=q} < 0 \qquad \forall q \in \mathcal{Q} \setminus \mathcal{F}
	\end{aligned}
\end{equation}
Next, we show that this theorem is fulfilled for the function
\begin{equation}
	f(q) = \ln \ln (e+\rho(q))
\end{equation}
with $\rho(q)$ being the squared distance from $q$ to the main diagonal of the original state space $\mathbb{Z}^3$:
\begin{equation}
	\rho(q) = q^\intercal q -  \frac{1}{3} \left( q^\intercal \mathbf{1} \right)^2
\end{equation}
It is readily checked that this function is well defined for the actual state space $\mathcal{Q} = \mathbb{Z}^3 \slash \langle \mathbf{1} \rangle$ of $(\qp_t)$ and fulfills the first condition in \eqref{eq::kendall_cond}.
The drift $\Delta f(q)$ is the expected change in $f$ during one evolution of $(\qp_t)$ and is independent of the control, since $(\qp_t)$ is as well. Because the entries in $(a_t)$ are from $\{0,1\}$, and by the virtue of $p=\frac{1}{2}$, the drift simply becomes
\begin{samepage}
\begin{align}
    \nonumber
	\Delta &f(q) =
	- f(q) + \frac{1}{8} \left[ f(q + e_1+e_2+e_3) + f(q) \right]
	\\
	&+ \frac{1}{8} \left[ f(q+e_1) + f(q+e_2) + f(q+e_3) \right]
	\\
	\nonumber
	&+ \frac{1}{8} \left[ f(q+e_1+e_2) + f(q+e_2+e_3) + f(q+e_3+e_1) \right]
\end{align}
\end{samepage}
with $e_i$ being the unit vector.

In our case, the fastest way to verify Kendall's theorem is via a geometrical interpretation of the state space. For $d=3$ we essentially move on a 2-dimensional plane with 3 axis as illustrated in \figref{fig::state_space}. Hence, we can identify $q$ with its radius $r$ and an angle $\phi$ shared with what was formerly the $(1,0,0)$ axis. In these coordinates (and using the "law of cosinus"), the drift becomes
\begin{align}
	\Delta f(q) &= -\frac{6}{8} \ln\ln \left( e + r^2 \right)
	\\
	\nonumber
	&+ \frac{1}{8} \sum_{m=0}^5
	\ln \ln \left(
		e + r^2 + 1 - 2r \cos(\phi + m \cdot  60^{\circ})
	\right)
\end{align}
Differentiating by $\phi$ yields maxima at $\phi = 30^\circ + z \cdot 60^\circ$, with $z \in \mathbb{Z}$, such that
\begin{samepage}
\begin{gather}
	\Delta f(q) \leq -\frac{3}{4} \ln\ln \left( e + r^2 \right)
	+
	\frac{1}{4} \ln \ln \left(
		e + r^2 + 1
	\right)
	\\
	\nonumber
	+
	\frac{1}{4} \ln \ln \left(
		e + r^2 + 1 - \sqrt{3}r
	\right)
	+
	\frac{1}{4} \ln \ln \left(
		e + r^2 + 1 + \sqrt{3}r
	\right)
\end{gather}
\end{samepage}
The RHS obviously tends to 0 as $r\to \infty$, because $r^2$ dominates the other terms in the $\ln$-function and the coefficients in front of $\ln$-functions add up to 0. 
On the other hand, differentiating the RHS by $r$ yields
\begin{equation}
\begin{aligned}
	-\frac{2}{  r^2  \ln \left( r^2 \right) }
	+
	\frac{1}{ \left( r^2 - \sqrt{3}r \right) \ln \left( r^2 - \sqrt{3}r \right) }
	&
	\\ 	
	+
	\frac{1}{ \left( r^2 + \sqrt{3}r \right) \ln \left( r^2 + \sqrt{3}r \right) }
	&> 0
\end{aligned}
\end{equation}
which is positive due to the convex nature of the occurring functions as $r \to \infty$.
This allows us to deduct as follows: in the direction of the largest possible values for $\Delta f(q)$, we have $\Delta f(q)$ tending to $0$ from the negative values (since the derivation by $r$ is positive). It follows that $\Delta f(q)$ must be negative for all $r \in \mathbb{R}_+$ past a certain threshold, leaving only a finite set of states for which $\Delta f(q) \geq 0$, and thus fulfilling \eqref{eq::kendall_cond}.
\\

To finish the prove (for Theorem \ref{th::one}), we have to show that the non-transience in the cases $d=2$ and $d=3$ is indeed a null-recurrence and that all these results stay true when recombining $(\qp_t)$ with $(\qc_t)$ to the original DTMC $(q_t)$.

Regarding the first part, note that if $d=2$, $(\qp_t)$ is a simple symmetric random walk with an additional probability to stay idle, due to the correspondence
\begin{equation}
    \begin{pmatrix}
        q^{(1)}_t \\ q^{(2)}_t
    \end{pmatrix}
    \operatorname{mod} \mathbf{1}
    \ \Longleftrightarrow \
    q^{(1)}_t - q^{(2)}_t
    =
    q^{(1)}_0 - q^{(2)}_0 + \sum_{\tau=0}^{t-1} a^{(1)}_\tau - a^{(2)}_\tau
\end{equation}
where $q^{(i)}_t$ means the $i$th entry of the state vector $\qp_t$.
Since idle time slots will strictly increase any passage times, the return time for $(\qp_t)$ must be larger than that of the simple symmetric random walk. But latter is well known to have infinite return time and thus null-recurrence follows for $(\qp_t)$.

The same argument can be applied to the case $d=3$, via the correspondence
\begin{equation}
    \begin{pmatrix}
        q^{(1)}_t \\ q^{(2)}_t \\ q^{(3)}_t
    \end{pmatrix}
    \operatorname{mod} \mathbf{1}
    \ \Longleftrightarrow \
    \begin{pmatrix}
        q^{(1)}_t - q^{(2)}_t
        \\
        q^{(1)}_t - q^{(3)}_t
    \end{pmatrix}
\end{equation}
As before, the entries on the RHS constitute 2 simple symmetric random walks with idle times, which imply null-recurrence for $(\qp_t)$.

Regarding the the combination of $(\qp_t)$ with $(\qc_t)$, note that $(\qc_t)$ is per definition a simple random walk (with idle times) on the half line. Furthermore, $(\qc_t)$ is (positive) recurrent due to the assumption $p < \e{m}$. And since both DTMCs are independent, their combination $(q_t)$ naturally yields a null-recurrent DTMC, which finishes the proof.

\section{Conclusion}
We could prove that a batch of $d$ synchronized queues with i.i.d. arrivals, set up in a time-discrete manner and operating under an optimal control is quasi-stable for $d \in \{2,3\}$ and unstable for $d \geq 4$. Here, quasi-stable and unstable refer to null-recurrence and transience of the corresponding DTMC, respectively. This implies that for $d \in \{2,3\}$, the queue states possibly becomes infinite in some time slots, but will not tend towards it.

As a parallel result, we could prove that the power sums over skewed, normalized binomial coefficients do diverge for $d \in \{2,3\}$ and converge for $d\geq 4$.

\section{Open Questions}

As already investigated for the 2-dimensional case in continuous-time by \cite{Latouche1981}, it stands to question, whether the queueing process becomes stable if the arrival process is assumed to be semi-stochastic. E.g. one could assume that in every consecutive interval of $T \gg 0$ time slots, the same amount of customers has arrived at each queue. Another option would be to slightly in- or decrease the arrival rates based on the excess in a manner that minimizes it.

\begin{figure}
	\centering
	\def\svgwidth{\linewidth}
	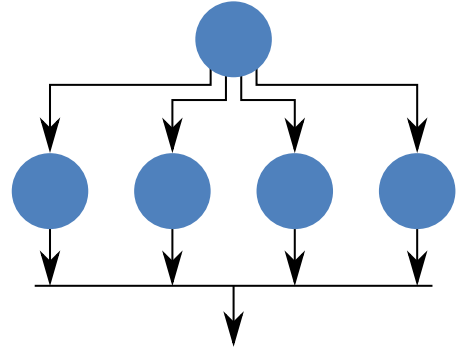
	\caption{Network, consisting of a conventional queue, that redirects its customers to different queues of a batch of synchronized queues.}
	\label{fig::mitigation}
\end{figure}

Furthermore, in a network context, the stochastics from the arrival processes might be mitigated entirely, before they reach the batch. See e.g. \figref{fig::mitigation}, where a controller can decide which of the synchronized queues to feed, while the only stochastic arrivals happen to the upper queue.

Finally, (though we assume it to be true) it stands to show that our results can be generalized to the continuous-time model and to the case, in which service of the batch requires more than one customer at certain queues (asymmetric batch-throughput).

\section*{Acknowledgment}
This work is part of and thereby funded by the DFG Priority Program 1914

\bibliographystyle{ieeetr}
\bibliography{oii}

\begin{thebibliography}{1}

\bibitem{Harrison1973}
J.~.~M. Harrison, ``{Assembly-like Queues},'' {\em Journal of Applied
  Probability}, 1973.

\bibitem{DeCuypere2014}
E.~{De Cuypere}, K.~{De Turck}, and D.~Fiems, ``{A Maclaurin-series expansion
  approach to multiple paired queues},'' {\em Operations Research Letters},
  2014.

\bibitem{Olvera-Cravioto2014}
M.~Olvera-Cravioto and O.~Ruiz-Lacedelli, ``{Parallel queues with
  synchronization},'' {\em arXiv}, 2014.

\bibitem{Gurvich2014}
I.~Gurvich and A.~Ward, ``{On the dynamic control of matching queues},'' {\em
  Stochastic Systems}, 2014.

\bibitem{Latouche1981}
G.~Latouche, ``{Queues with paired customers},'' {\em Journal of Applied
  Probability}, 1981.

\bibitem{Cusick1989}
T.~W. Cusick, ``{Recurrences for sums of powers of binomial coefficients},''
  {\em Journal of Combinatorial Theory, Series A}, 1989.

\bibitem{Kendall1951}
D.~G. Kendall, ``{On non-dissipative Markoff chains with an enumerable infinity
  of states},'' {\em Mathematical Proceedings of the Cambridge Philosophical
  Society}, 1951.

\end{thebibliography}
%\bibliography{C:/Users/michselbst/Desktop/Literature/library}

\end{document}